\documentclass[a4paper,12pt]{article}
\usepackage[utf8]{inputenc}
\usepackage{babel}
\usepackage{fontenc}
\usepackage{graphicx}
\usepackage[utf8]{inputenc}
\usepackage[T1]{fontenc}
\usepackage{lmodern}
\usepackage{ifxetex,ifluatex}
\if\ifxetex T\else\ifluatex T\else F\fi\fi T%
\usepackage{fontspec}
\else
\usepackage[T1]{fontenc}
\usepackage[utf8]{inputenc}
\usepackage{lmodern}
\fi
\usepackage{amsmath, amssymb, amsfonts}
\usepackage{amsthm}
\usepackage{listings}

\usepackage[top=1.1 in, bottom=1.4 in, left=1.1 in, right=1.1 in]{geometry}

\newtheorem{theorem}{Theorem}[section]
\newtheorem{lemma}{Lemma}[section]

\theoremstyle{definition}
\newtheorem{definition}{Definition}[section]

\theoremstyle{definition}

\theoremstyle{definition}
\newtheorem{example}{Example}[section]
\theoremstyle{definition}
\newtheorem{prop}{Proposition}[section]
\theoremstyle{definition}
\newtheorem{remark}{Remark}[section]
\usepackage{tikz-cd}
\usetikzlibrary{arrows,chains,matrix,positioning,scopes}
\makeatletter
\tikzset{join/.code=\tikzset{after node path={%
			\ifx\tikzchainprevious\pgfutil@empty\else(\tikzchainprevious)%
			edge[every join]#1(\tikzchaincurrent)\fi}}}
\makeatother
\tikzset{>=stealth',every on chain/.append style={join},
	every join/.style={->}}
\tikzstyle{labeled}=[execute at begin node=$\scriptstyle,
execute at end node=$]
\title{Irreducible representation of Plesken Lie algebras}
\author{S. N. Arjun\footnote{Corresponding author} \ and P. G. Romeo}
\date{}

\begin{document}
 \maketitle
 \begin{center}
 	 \textit{Department of Mathematics},\\ \textit{Cochin University of Science and Technology},\\
	\textit{Kochi, Kerala, India, 682022}\\
	\textit{arjunsnmaths1996@gmail.com} \\
	\textit{romeo\_parackal@yahoo.com}
 \end{center}
\begin{abstract}
	In \cite{Plesken}, Cohen and Taylor introduced Plesken Lie algebra as certain Lie algebra constructed using finite groups. Arjun and Romeo described the linear representation of these Lie algebras induced from group representation in \cite{Romeo}. Hence the authors posed the question as to what are the irreducible representations of Plesken Lie algebras and describes the irreducible representations of cover of Plesken Lie algebras.
\end{abstract}
\section{Introduction}
Lie groups, Lie algebras, and their representations constitute a fascinating field of study with a wide array of practical applications. The Lie groups has its roots in the work of Sophus Lie, who studied certain transformation groups that are now known as Lie groups. The Lie algebras are certain vector spaces that possess the Lie bracket operation, which satisfies specific algebraic properties. Moreover, a significant Lie group-Lie algebra correspondence enables the interrelation between a Lie group and its corresponding Lie algebra, or vice versa. The Lie algebra, being a linear object, is more immediately accessible than the group.
\par A few years back, W. Plesken and Arjeh M. Cohen constructed a Lie algebra from a finite group, and it is known as Plesken Lie algebra. A Plesken Lie algebra denoted $\mathcal{L}(G)$ of a  group $G$ over $\mathbb{F}$   is the linear span of elements $\hat{g} = g-{g}^{-1} \in \mathbb{F}G$ together with the Lie bracket $[\hat{g}, \hat{h}] = \hat{g}\hat{h} - \hat{h}\hat{g}$. Subsequently, Cohen and Taylor discussed the structure of Plesken Lie algebras and explicitly determine the groups for which the Plesken Lie algebra simple and semisimple over complex field. More recently, Romeo and Arjun discussed the representations of Plesken Lie algebras, Plesken Lie algebra modules and the Plesken Lie algebra representations induced from group representation. Further they determined irreducible Plesken Lie algebra representations.
\par The Plesken Lie algebra representations which they found in \cite{Romeo} are the representations of $\mathcal{L}(G)$ which are linearly extended from the representations of $G$. One issue they encounter is that preservation of irreducibility is not guaranteed. Thus our aim is to find the irreducible representations of $\mathcal{L}(G)$. We partially addressed this issue using the concept of projective representations, leading to the establishment of the following theorem.
\begin{theorem}
Let $\mathcal{E}$  a finite dimensional Lie algebra  and $\mathcal{L}(H)$ be a central subalgebra of $\mathcal{E}$ such that $\mathcal{E}$ is a cover of $\mathcal{E}/\mathcal{L}(H)$. Then there is a bijection between the sets $Irr(\mathcal{E})$ and $\displaystyle\bigcup_{[\alpha]\in H^2(\mathcal{E}/\mathcal{L}(H), \mathbb{C})} Irr^{\alpha}(\mathcal{E}/\mathcal{L}(H))$.
\end{theorem}
where $Irr(\mathcal{E})$ denotes the set of irreducible linear representations of $\mathcal{E}$ and
 \\
 $Irr^{\alpha}(\mathcal{E}/\mathcal{L}(H))$ is the set of irreducible $\alpha$-representations of $\mathcal{E}/\mathcal{L}(H)$ where $\alpha$ ranges over $H^2(\mathcal{E}/\mathcal{L}(H), \mathbb{C})$.
\par The theory of projective representations involves understanding homomorphisms from a Lie algebra to the quotient Lie algebra $\mathfrak{gl}(V)/\{kI_V: k \in \mathbb{C} \}$, called the projective Lie algebra. These representations appear in the study of linear representation of Lie algebras as in the above theorem. By definition, every linear representation of a Lie algebra is projective, but the converse is not true. Therefore, understanding the projective representation is more more complicated. Recall, the Schur multilplier of a finite dimensional Lie algbera $\mathcal{L}$ is the second cohomology group $H^2(\mathcal{L}, \mathbb{C})$, where $\mathbb{C}$ is a trivial $\mathcal{L}$-module(cf. \cite{Batten}). The Schur multilplier plays an important role in understanding projective representations. By definition, every projective representation $\Phi$ of $\mathcal{L}(G)$ is associated with a 2-cocycle $\alpha : \mathcal{L}(G) \times \mathcal{L}(G) \to \mathbb{C}$ such that $[\Phi(x), \Phi(y)] = \alpha(x,y)I_V + \Phi([x, y])$ for all $x, y \in \mathcal{L}(G)$.  In this case, we say that $\Phi$ is an $\alpha$-representation. Conversely, for every 2-cocycle $\alpha$ of $\mathcal{L}(G)$, we can find an $\alpha$-representation of $\mathcal{L}(G)$. Thus the first thing to understand the projective representation is to describe the 2-cocycles of $\mathcal{L}(G)$ upto cohomologous, that is, to understand the Schur multiplier of $\mathcal{L}(G)$ and then constructing $\alpha$-representations of $\mathcal{L}(G)$ for all $[\alpha] \in H^2(\mathcal{L}(G), \mathbb{C})$, where $[\alpha]$ denotes the cohomology class of $\alpha$. All these observations and results are described in Section \ref{projsec}.

\par The next interesting thing is that the central extensions play an important role in the study of projective representations of a Lie algebra. The key concept that appear in the context is cover.  The cover of a Lie algebra introduced and described by P. G. Batten in her Ph. D Thesis (\cite{Batten}).
Section 3 provides an exposition on the cover of a Plesken Lie algebra and established a one-to-one correspondence between irreducible linear representations of the cover of $\mathcal{L}(G)$ and irreducible projective representations of $\mathcal{L}(G)$.

 Throughout this paper $\mathcal{L}(G)$ denote the Plesken Lie algebra of a finite group $G$ over the field of complex numbers $\mathbb{C}$ and $V$ is a finite dimensional complex vector space, otherwise stated.\\

 \section{Projective representations of $\mathcal{L}(G)$}\label{projsec}
 In this section, we will introduce the projective representation of Plesken Lie algebras and establish it's relation between the second cohomology group.
 \par A \textit{projective representation} of a Plesken Lie algebra $\mathcal{L}(G)$ is a Lie algebra homomorphism $\phi: \mathcal{L}(G) \to \mathfrak{pgl}(V)$ (for some finite dimensional vector space $V$), where $\mathfrak{pgl}(V)$ is the quotient Lie algebra $\mathfrak{gl}(V)/\{kI_V: k \in \mathbb{C} \}$.
\par Every linear representation of $\mathcal{L}(G)$ is a projective representation. For, consider a linear representation $\rho : \mathcal{L}(G) \to \mathfrak{gl}(V)$  of $\mathcal{L}(G)$ and the natural homomorphism $\pi : \mathfrak{gl}(V) \to \mathfrak{pgl}(V)$, then the composition $\pi \circ \rho  : \mathcal{L}(G) \to  \mathfrak{pgl}(V)$ is a projective representation.
\par The following proposition gives a characterization for the projective representation of Lie algebras.

\begin{prop}\cite{Arjun}
	Let  $G$ be a finite group and $\phi$ a projective representation  of $\mathcal{L}(G)$ on $V$. Then there is a linear map $\Phi : \mathcal{L}(G) \to \mathfrak{gl}(V)$ and a bilinear map $\mu :\mathcal{L}(G)\times \mathcal{L}(G) \to \mathbb{C}$ such that
	\begin{equation}
		[\Phi(\hat{x}),\Phi(\hat{y})]=\mu(\hat{x}, \hat{y})I_V+\Phi([\hat{x}, \hat{y}]) \ \text{ for all }\ \hat{x}, \hat{y} \in \mathcal{L}(G).
	\end{equation}
	Conversely, if there is a linear map $\Phi$ and a bilinear map $\mu$ satisfying $(1)$, then  $\pi \circ \Phi : \mathcal{L}(G) \to \mathfrak{pgl}(V)$  where $\pi : \mathfrak{gl}(V) \to \mathfrak{pgl}(V)$ is the canonical homomorphism, is a projective representation of $\mathcal{L}(G)$. 
\end{prop}
The details are given by the following commutative diagram.
\[
  \begin{tikzcd}
    \mathcal{L}(G) \arrow{r}{\Phi} \arrow[swap]{dr}{\phi} & \mathfrak{gl}(V) \arrow{d}{\pi} \\
     & \mathfrak{pgl}(V)
  \end{tikzcd}
\]

Now we can define the projective representation of $\mathcal{L}(G)$ as follows:
\begin{definition}
	A linear map $\Phi: \mathcal{L}(G) \to \mathfrak{gl}(V)$ together with a bilinear map $\mu :\mathcal{L}(G) \times \mathcal{L}(G) \to \mathbb{C}$ satisfying the condition
	\begin{center}
		$[\Phi(\hat{x}),\Phi(\hat{y})]= \mu(\hat{x}, \hat{y})I_V + \Phi([\hat{x},\hat{y}])$.
	\end{center}
	is called a projective representation of $\mathcal{L}(G)$.
\end{definition}
Next we proceed to explore the connection between projective representations and cohomology groups.
\par For, recall the  second cohomology group of a finite dimensional Lie algebra (cf.\cite{Batten}). Let $\mathcal{L}(G)$ be a Plesken Lie algebra of $G$ over $\mathbb{C}$. The set of 2-cocycles is given by
\begin{center}
	$Z^2(\mathcal{L}(G), \mathbb{C}) = \{ f: \mathcal{L}(G) \times  \mathcal{L}(G) \to \mathbb{C} : f \text{ is bilinear and } f([\hat{x}, \hat{y}], \hat{z})$\\\ \hspace{1in} $+f([\hat{y}, \hat{z}], \hat{x})+f([\hat{z},\hat{x}],\hat{y}) = 0  \ \forall \hat{x}, \hat{y}, \hat{z} \in \mathcal{L}(G)\}$
\end{center}
and the set of 2-coboundaries are 
\begin{center}
	$B^2(\mathcal{L}(G), \mathbb{C})=\{ f: \mathcal{L}(G) \times  \mathcal{L}(G) \to \mathbb{C} : f \text{ is bilinear and there exists  } $\\ \hspace{1in} $\sigma : \mathcal{L}(G) \to \mathbb{C} \text{ such that } f(\hat{x}, \hat{y})= -\sigma([\hat{x}, \hat{y}]) \}$.
\end{center}
Then the second cohomology group of $\mathcal{L}(G)$ is given by
\begin{center}
	$H^2(\mathcal{L}(G), \mathbb{C}) = \frac{Z^2(\mathcal{L}(G), \mathbb{C})}{B^2(\mathcal{L}(G), \mathbb{C})}.$
\end{center}
 and is called the $multiplier$ of $\mathcal{L}(G)$ (see \cite{Batten}).
\par Note that two 2-cocycles $\alpha_1$ and $\alpha_2$ are said to be $cohomologous$ (ie., they have the same cohomology class) if there exists a linear map $\sigma : \mathcal{L}(G) \to  \mathbb{C}$ such that
\begin{center}
	$\mu_2(\hat{x}, \hat{y}) - \mu_1(\hat{x}, \hat{y}) = -\sigma([\hat{x}, \hat{y}])$
\end{center}
 Let $\Phi: \mathcal{L}(G)\to \mathfrak{gl}(V)$ a projective representation , then for any $\Phi(\hat{x}),\Phi(\hat{y}), \Phi(\hat{z}) \in \mathfrak{gl}(V)$,
\begin{equation}
	[[\Phi(\hat{x}),\Phi(\hat{y})],\Phi(\hat{z})]+ [[\Phi(\hat{y}),\Phi(\hat{z})],\Phi(\hat{x})]+ [[\Phi(\hat{z}),\Phi(\hat{x})],\Phi(\hat{y})]= 0
\end{equation}
(by Jacobi identity in $\mathfrak{gl}(V)$),
\begin{equation}
	\begin{split}
		[[\Phi(\hat{x}),\Phi(\hat{y})],\Phi(\hat{z})]&= [\mu(\hat{x},\hat{y})I_V+\Phi([\hat{x}, \hat{y}]), \Phi(\hat{z})]\\
		&= \mu(\hat{x}, \hat{y})[I_V, \Phi(\hat{z})]+ [\Phi([\hat{x}, \hat{y}]), \Phi(\hat{z})]\\
		&= [\Phi([\hat{x}, \hat{y}]), \Phi(\hat{z})]\\
		&= \mu([\hat{x},\hat{y}],\hat{z})I_V+ \Phi([[\hat{x},\hat{y}],\hat{z}])
	\end{split}
\end{equation}
and 
\begin{equation}
	\begin{split}
		[[\Phi(\hat{x}),\Phi(\hat{y})],\Phi(\hat{z})]+& [[\Phi(\hat{y}),\Phi(\hat{z})],\Phi(\hat{x})]+ [[\Phi(\hat{z}),\Phi(\hat{x})],\Phi(\hat{y})]\\
		&= \mu([\hat{x}, \hat{y}],\hat{z})I_V+ \Phi([[\hat{x},\hat{y}],\hat{z}])+\mu([\hat{y},\hat{z}],\hat{x})I_V\\
		&\quad + \Phi([[\hat{y},\hat{z}],\hat{x}])+\mu([\hat{z},\hat{x}],\hat{y})I_V+ \Phi([[\hat{z},\hat{x}],\hat{y}]) \\
		&=(\mu([\hat{x},\hat{y}],\hat{z})+\mu([\hat{y},\hat{z}],\hat{x})+\mu([\hat{z},\hat{x}],\hat{y}))I_V\\
		&\quad +\Phi([[\hat{x},\hat{y}],\hat{z}]+[[\hat{y},\hat{z}],\hat{x}]+[[\hat{z},\hat{x}],\hat{y}])\\
		&=(\mu([\hat{x},\hat{y}],\hat{z})+\mu([\hat{y},\hat{z}],\hat{x})+\mu([\hat{z},\hat{x}],\hat{y}))I_V
	\end{split}
\end{equation}
Then from $(2)$,
\begin{equation}
	\mu([\hat{x},\hat{y}],\hat{z})+\mu([\hat{y},\hat{z}],\hat{x})+\mu([\hat{z},\hat{x}],\hat{y})=0
\end{equation}
$\mu$ is a bilinear map and satisfies the 2-cocycle condition. That is, $\mu \in Z^2(\mathcal{L}(G),\mathbb{C})$.  Thus the projective representation $\Phi$ is also referred to as an \textit{$\mu$-representation } on the vector space $V$.
\subsection{Projectively equivalent $\mu$-representations of $\mathcal{L}(G)$}
Here we proceed to define certain equivalence of projective representations of $\mathcal{L}(G)$.
\begin{definition}
	Let $\Phi_1$ be an $\mu_1$-representation and $\Phi_2$ be an $\mu_2$-representation of $\mathcal{L}(G)$ on the complex vector spaces $V$ and $W$ respectively. Then $\Phi_1$ and $\Phi_2$ are projectively equivalent if there exists an isomorphism $f : V \to W$ and a linear map $\delta: \mathcal{L}(G) \to \mathbb{C}$ such that
	\begin{center}
		$\Phi_2(\hat{x}) = f\circ \Phi_1(\hat{x})\circ f^{-1} + \delta(\hat{x})I_W$ for all $\hat{x} \in \mathcal{L}(G)$
	\end{center}
	
	If $\delta(\hat{x})= 0$ for all $\hat{x} \in \mathcal{L}(G)$, then $\Phi_1$ and $\Phi_2$ are linearly equivalent.
\end{definition}
Next we recall the following theorems describing the correspondence between two equivalent projective representations of $\mathcal{L}(G)$ and $H^2(\mathcal{L}(G), \mathbb{C})$. 
\begin{theorem}\cite{Arjun}
	Let $\Phi_1$ be an $\mu_1$-representation and $\Phi_2$ be an $\mu_2$-representation of $\mathcal{L}(G)$. If $\Phi_1$ is projectively equivalent to $\Phi_2$, then $\mu_1$ and $\mu_2$ are cohomologous.
\end{theorem}
\begin{theorem}\cite{Arjun}\label{theorem}
	Let $\Phi_1$ be an $\mu_1$-representation of $\mathcal{L}(G)$ and $\mu_2$ is a 2-cocycle cohomologous to $\mu_1$, then there exists an $\mu _2$-representation $\Phi_2$ of $\mathcal{L}(G)$ which  is projectively equivalent to $\Phi_1$.
\end{theorem}
Thus, any projective representations of $\mathcal{L}(G)$ up to projective equivalence defines an element of the second cohomology group $H^2(\mathcal{L}(G), \mathbb{C})$.

\subsection{Central extensions and covers}\label{coversec}
\par Next we introduce central extension of a Plesken Lie algebra $\mathcal{L}(G)$.
\begin{definition} 
	$G$ and $H$ are two finite groups such that $\mathcal{L}(G)$ and $\mathcal{L}(H)$ are the corresponding Plesken Lie algebras of which $\mathcal{L}(H)$ is abelian. Then $(\mathcal{E}; f, g)$ is an extension of $\mathcal{L}(G)$  by $\mathcal{L}(H)$	if there exists a Lie algebra $\mathcal{E}$ such that the following is a short exact sequence : 
	\begin{center}
		$0 \to \mathcal{L}(H)\xrightarrow{f} \mathcal{E} \xrightarrow{g} \mathcal{L}(G) \to 0$.
	\end{center}
\end{definition}
An extension $(\mathcal{E};f,g)$ is \textit{central} if $f(\mathcal{L}(H)) \subseteq Z(\mathcal{E})$ where $Z(\mathcal{E})$ is the center of $\mathcal{E}$.
\begin{example}\label{central} Consider the subgroups $G=\{ \begin{pmatrix}1&0&y\\0&1&z\\0&0&1\end{pmatrix} : y, z\in \mathbb{Z}_p\}$ and $H=\{ \begin{pmatrix}1&0&y\\0&1&0\\0&0&1\end{pmatrix} : y \in \mathbb{Z}_p\}$ of the Heisenberg group $\mathbb{H}(\mathbb{Z}_p)$. 
	Then the Plesken Lie algebras of $G$ and $H$ are	
	\begin{equation*}
		\begin{split}
			\mathcal{L}(G)&= span_{\mathbb{C}}\{ \begin{pmatrix}0&0&y\\0&0&z\\0&0&0\end{pmatrix} : y, z\in \mathbb{Z}_p\}
			\text{ and } \\
			\mathcal{L}(H)&= span_{\mathbb{C}}\{ \begin{pmatrix}0&0&y\\0&0&0\\0&0&0\end{pmatrix} : b\in \mathbb{Z}_p\}
		\end{split}
	\end{equation*}
	respectively. Consider 
	\begin{center}
		$\mathfrak{h} = \mathcal{L}(\mathbb{H}(\mathbb{Z}_p))= span_{\mathbb{C}}\{ \begin{pmatrix}0&x&y\\0&0&z\\0&0&0\end{pmatrix} : x, y, z\in \mathbb{Z}_p\}$,
	\end{center}
	$f : \mathcal{L}(H)\to \mathfrak{h}$  is the identity inclusion and $g : \mathfrak{h} \to \mathcal{L}(G)$ is given by
	\begin{center}
		$g(\begin{pmatrix}0&x&y\\0&0&z\\0&0&0\end{pmatrix}) = \begin{pmatrix}0&0&x\\0&0&z\\0&0&0\end{pmatrix}$.
	\end{center}
	Then $f$ and $g$ are Lie algebra homomorphisms and $ker(g) = Im(f)$. Also $f(\mathcal{L}(H)) = \mathcal{L}(H) = Z(\mathfrak{h})$, the center of $\mathfrak{h}$. Thus  $0 \to \mathcal{L}(H)\xrightarrow{f} \mathfrak{h} \xrightarrow{g} \mathcal{L}(G) \to 0$ is a central extension.


\end{example}
\begin{definition}
Two central extensions
\begin{center}
	$0\to \mathcal{L}(H) \xrightarrow{f} \mathcal{E} \xrightarrow{g} \mathcal{L}(G) \to 0$
\end{center}
and
\begin{center}
	$0\to \mathcal{L}(H) \xrightarrow{f'} \mathcal{E}' \xrightarrow{g'} \mathcal{L}(G) \to 0$
\end{center}
of $\mathcal{L}(G)$ by $\mathcal{L}(H)$ are equivalent if there exists a Lie algebra homomorphism $\phi : \mathcal{E} \to \mathcal{E}'$ such that the following diagram commutes:
\begin{equation}
	\begin{tikzpicture}
		\matrix (m) [matrix of math nodes, row sep=3em, column sep=3em]
		{ 0 & \mathcal{L}(H) & \mathcal{E}  & \mathcal{L}(G) & 0 \\
			0 & \mathcal{L}(H) & \mathcal{E}'  & \mathcal{L}(G) & 0 \\ };
		{ [start chain] \chainin (m-1-1);
			\chainin (m-1-2);
			{ [start branch=A] \chainin (m-2-2)
				[join={node[right,labeled] {id}}];}
			\chainin (m-1-3) [join={node[above,labeled] {f}}];
			{ [start branch=B] \chainin (m-2-3)
				[join={node[right,labeled] {\phi}}];}
			\chainin (m-1-4) [join={node[above,labeled] {g}}];
			{ [start branch=C] \chainin (m-2-4)
				[join={node[right,labeled] {id}}];}
			\chainin (m-1-5); }
		{ [start chain] \chainin (m-2-1);
			\chainin (m-2-2);
			\chainin (m-2-3) [join={node[above,labeled] {f'}}];
			\chainin (m-2-4) [join={node[above,labeled] {g'}}];
			\chainin (m-2-5); }
	\end{tikzpicture}
\end{equation}
\end{definition}
\begin{remark}
	If such a homomorphism $\phi$ exists, then it must be an isomorphism ( see \cite{Batten}).
\end{remark}

The following theorem gives us the relation between the multiplier $H^2(\mathcal{L}(G), \mathbb{C})$ and central extensions of $\mathcal{L}(G)$ by $\mathbb{C}$.
\begin{theorem}\cite{Arjun}
	Let $\mathcal{L}(G)$ be a Plesken Lie algebra of a finite group $G$. Then there is a bijective correspondence between $H^2(\mathcal{L}(G), \mathbb{C})$ and equivalence class of central extensions of $\mathcal{L}(G)$ by $\mathbb{C}$.
\end{theorem}

Now it is easy to observe the following:
\begin{itemize}
	\item The classical First and Second Whitehead Lemmas state that the first, respectively second, cohomology group of a finite-dimensional semisimple Lie algebra with coefficients in any finite-dimensional module vanishes (see \cite{whitehead}). Thus we can conclude that for semisimple Plesken Lie algebras, there exists no non-trivial projective representations as well as non-trivial central extensions.
	\item We showed the bijective correspondence between set of equivalence class of central extensions of $\mathcal{L}(G)$ by $\mathbb{C}$ and $H^2(\mathcal{L}(G), \mathbb{C})$ and the bijective correspondence between set of equivalence classes of projective representations of $\mathcal{L}(G)$ and  $H^2(\mathcal{L}(G), \mathbb{C})$. Hence we infer that there is a bijective correspondence between equivalent projective representations of $\mathcal{L}(G)$ and equivalent central extensions of $\mathcal{L}(G)$ by $\mathbb{C}$.
\end{itemize}
\subsection{Cover of $\mathcal{L}(G)$}
Next we proceed to describe the linear representation of the cover of a Plesken Lie algebra $\mathcal{L}(G)$ using projective representation of $\mathcal{L}(G)$.
\par First we recall some definitions and results related to covers of Lie algebra (see \cite{Batten}).
\begin{definition}
	A pair of Lie algebras $(\mathcal{E}, \mathcal{M})$ is said to be a defining pair of $L$ if 
	\begin{enumerate}
		\item $L \cong \mathcal{E}/\mathcal{M}$
		\item $\mathcal{M} \subseteq Z(\mathcal{E}) \cup [\mathcal{E}, \mathcal{E}]$
	\end{enumerate}
\end{definition}

\begin{lemma}\label{maximal}
	Let $L$ be a Lie algebra of finite dimension $n$ and $\mathcal{E}$ be the first term in a defining pair of $L$. Then $ dim(\mathcal{E}) \le \frac{n(n+1)}{2}$.
\end{lemma}
\begin{definition}
	$(\mathcal{E}, \mathcal{M})$ is called maximal defining pair if the dimension of $\mathcal{E}$ is maximal. For these maximal defining pairs, $\mathcal{E}$ is called cover of $L$ and $\mathcal{M}$, which is unique since it is abelian, is called the multiplier for $L$. 
\end{definition}
We can alternatively define cover of a Lie algebra $L$ using central extensions. In particular for Plesken Lie algebras, the definition of a cover is as follows:
\begin{definition}
	Let $\mathcal{L}(G)$ and $\mathcal{L}(H)$ be Plesken Lie algebras of which $\mathcal{L}(H)$ is abelian. A Plesken Lie algebra $\mathcal{E}$ is said to be cover of  $\mathcal{L}(G)$ if
	\begin{enumerate}
		\item $0 \to \mathcal{L}(H) \xrightarrow{f} \mathcal{E} \xrightarrow{g} \mathcal{L}(G) \to 0$ is a central extension
		\item $Ker(g) \subseteq [\mathcal{E}, \mathcal{E}]$ and $Ker(g) \cong H^2(\mathcal{L}(G), \mathbb{C})$
		\item $dim(\mathcal{E})$ is maximal.
	\end{enumerate}
\end{definition}
\par In the following example, we recall the central extension described in Example \ref{central} and explain the cover of $\mathcal{L}(G)$ where $G=\{ \begin{pmatrix}1&0&b\\0&1&c\\0&0&1\end{pmatrix} : b, c\in \mathbb{Z}_p\}$.
\begin{example}\label{cover}
	Consider the central extension  $0 \to \mathcal{L}(H)\xrightarrow{f} \mathfrak{h} \xrightarrow{g} \mathcal{L}(G) \to 0$ of $\mathcal{L}(G)$ by $\mathcal{L}(H)$ as in Example \ref{central}. 
	Here 
	\begin{center}
		$[\mathfrak{h}, \mathfrak{h}] =  span_{\mathbb{C}}\{ \begin{pmatrix}0&0&b\\0&0&0\\0&0&0\end{pmatrix} : b\in \mathbb{Z}_p\}$
	\end{center}
	and thus $Ker(g) = [\mathfrak{h}, \mathfrak{h}]$. Also 
	\begin{equation*}
		\begin{split}
			dim(Ker(g))= dim(Im(f))=dim(\mathcal{L}(H)) = 1	,	   
		\end{split}
	\end{equation*}
	
	$dim(H^2(\mathcal{L}(G), \mathbb{C})) = 1$ (see Lemma 2.5, \cite{Peyman}) and since every one dimensional abelian Plesken Lie algebras are isomorphic, $Ker(g) \cong H^2(\mathcal{L}(G), \mathbb{C})$. Also $dim(\mathfrak{h}) = 3$ which is maximal due to lemma \ref{maximal}.  Hence, $\mathfrak{h}$ is a cover of $\mathcal{L}(G)$.
\end{example}
Next we investigate when it will be possible to find a cover $\mathcal{E}$ of $\mathcal{{L}}(G)$  such that given a projective representation of $\mathcal{L}(G)$, there corresponds a linear representation of $\mathcal{E}$.

\par For, recall the 5 sequence of cohomology for central extensions of Plesken Lie algebras also known as \textit{the Hochschild - Serre spectral sequence of low dimensions} (see \cite{Batten}). If $I$ is an ideal of the Plesken Lie algebra $\mathcal{L}(G)$, then it is easy to observe that $\mathcal{L}(G)/I$ inherits a natural Lie algebra structure from $\mathcal{L}(G)$, and there is an exact sequence of Lie algebra homomorphisms
\begin{equation}\label{HSSsequence}
	0 \to I \xrightarrow{f}\mathcal{L}(G) \xrightarrow{g} \mathcal{L}(G)/I \to 0.
\end{equation}
and $s : \mathcal{L}(G)/I \to \mathcal{L}(G)$ is a the section of $g$. For an $\mathcal{L}(G)$-module $A$ the sequence

\begin{equation}
	Hom(\mathcal{L}(G)/I, A) \xrightarrow{Inf_1} Hom(\mathcal{L}(G), A) \xrightarrow{Res} Hom(I, A) \xrightarrow{Tra} H^2(\mathcal{L}(G), A) \xrightarrow{Inf_2} H^2(\mathcal{L}(G)/I, A)
\end{equation}
is exact and is called the Hochschild - Serre spectral sequence (Theorem 3.1, cf. \cite{Batten})
, where $Inf_1$ and $Inf_2$ are inflation maps, $Res$ is the restriction map and $Tra : Hom(I, A) \rightarrow H^2(\mathcal{L}(G), A)$ is the Transgression map and is defined by
\begin{center}
	$Tra(\chi) = [\chi \circ \beta]$
\end{center}
for $\chi \in Hom(I, A)$, where $\beta(x, y) = [s(\hat{x}), s(\hat{y})] - s([\hat{x}, \hat{y}])$; $s$ is the section of $g$ in (\ref{HSSsequence}).
\par The projective representation of $\mathcal{L}(G)$ leads to a linear representation of the cover of $\mathcal{L}(G)$, as outlined in the following theorem.
 
\begin{theorem}\label{linearrepcover}
	Let $\mathcal{L}(G)$ be a Plesken Lie algebra and $\mathcal{E}$ be its cover. Suppose $\Phi : \mathcal{L}(G) \to \mathfrak{gl}(V)$ is $\mu$-representation of $\mathcal{L}(G)$. Let $[\mu] \in Im(Tra)$. Then there is a Lie algebra homomorphism $ \Gamma : \mathcal{E} \to \mathfrak{gl}(V)$  such that $\Gamma(\hat{h})$ is a scalar multiple of the identity transformation on $V$ for any $\hat{h} \in \mathcal{L}(H)$.
\end{theorem}
\begin{proof}
	Since $\Phi : \mathcal{L}(G) \to \mathfrak{gl}(V)$ is $\mu$-representation of $\mathcal{L}(G)$, for any $x, y \in \mathcal{L}(G)$
	\begin{center}
		$[\Phi(\hat{x}), \Phi(\hat{y})] = \mu(\hat{x}, \hat{y})I_V+\Phi([\hat{x}, \hat{y}])$.
	\end{center}
	Since $\mathcal{E}$ is the cover of $\mathcal{L}(G)$, there is a central extension $0 \to \mathcal{L}(H) \xrightarrow{f} \mathcal{E} \xrightarrow{g} \mathcal{L}(G) \to 0$ and  $[\mu] \in Im(Tra)$ implies that there exists $\chi \in Hom(\mathcal{L}(H), \mathbb{C})$ such that $\mu(\hat{x}, \hat{y}) = \chi(\beta(\hat{x}, \hat{y}))$.
	We have, for any $\hat{x} \in \mathcal{E}$, $\hat{x}= \hat{h} + s(\hat{y})$ for some $\hat{h} \in \mathcal{L}(H)$ and $\hat{y} \in \mathcal{L}(G)$ (where $s : \mathcal{L}(G) \to \mathcal{E}$ is the section of $g$). Define $\Gamma : \mathcal{E} \to \mathfrak{gl}(V)$ by
	\begin{center}
		$\Gamma(\hat{x}) = \Gamma(\hat{h} +s(\hat{y})) = \chi(\hat{h})I_V+\Phi(\hat{y})$
	\end{center}
	then $\Gamma$ is linear and for any $\hat{x_1} = \hat{h_1} +s(\hat{y_1}), \hat{x_2} = \hat{h_2} +s(\hat{y_2}) \in \mathcal{E}$,
	\begin{equation}\label{Gamma1}
		\begin{split}
			\Gamma([\hat{x_1}, \hat{x_2}])&= \Gamma([\hat{h_1} +s(\hat{y_1}),  \hat{h_2} +s(\hat{y_2})])\\
			&= \Gamma([s(\hat{y_1}), s(\hat{y_2})])\\
			&=\Gamma(\beta(\hat{y_1}, \hat{y_2}) + s([\hat{y_1}, \hat{y_2}]))\\
			&= \chi(\beta(\hat{y_1}, \hat{y_2}))I_V+\Phi([\hat{y_1}, \hat{y_2}])\\
			&= \mu(\hat{y_1}, \hat{y_2})I_V+ \Phi([\hat{y_1}, \hat{y_2}])\\
			&=[\Phi(\hat{y_1}), \Phi(\hat{y_2})]
		\end{split}
	\end{equation}
	and 
	\begin{equation}\label{Gamma2}
		\begin{split}
			[\Gamma(\hat{x_1}), \Gamma(\hat{x_2})]&=[\chi(\hat{h_1})I_V+\Phi(\hat{y_1}), \chi(\hat{h_2})I_V+\Phi(\hat{y_2})]\\
			&= [\Phi(\hat{y_1}), \Phi(\hat{y_2})].
		\end{split}
	\end{equation}
	Equations (\ref{Gamma1}) and (\ref{Gamma2}) gives $\Gamma$ is a Lie algebra homomorphism.
	Also for any $\hat{h} \in \mathcal{L}(H)$, $\Gamma(\hat{h})$ is a scalar multiple of the identity transformation on $V$.
\end{proof}

Conversely, given a linear representation of cover of $\mathcal{L}(G)$, then it is possible to obtain a projective representation of $\mathcal{L}(G)$. First we prove the following lwmma.

 \begin{lemma}\label{projlemma}
 Consider the commutative diagram of Lie algebra homomorphisms
 \begin{equation}
	\begin{tikzpicture}
		\matrix (m) [matrix of math nodes, row sep=3em, column sep=3em]
		{ 0 & \mathcal{L}(H) & \mathcal{E}  & \mathcal{L}(G) & 0 \\
			0 & \mathcal{L}(H') & \mathcal{E}'  & \mathcal{L}(G') & 0 \\ };
		{ [start chain] \chainin (m-1-1);
			\chainin (m-1-2);
			{ [start branch=A] \chainin (m-2-2)
				[join={node[right,labeled] {\phi}}];}
			\chainin (m-1-3) [join={node[above,labeled] {f}}];
			{ [start branch=B] \chainin (m-2-3)
				[join={node[right,labeled] {\psi}}];}
			\chainin (m-1-4) [join={node[above,labeled] {g}}];
			
			\chainin (m-1-5); }
		{ [start chain] \chainin (m-2-1);
			\chainin (m-2-2);
			\chainin (m-2-3) [join={node[above,labeled] {f'}}];
			\chainin (m-2-4) [join={node[above,labeled] {g'}}];
			\chainin (m-2-5); }
	\end{tikzpicture}
\end{equation}
such that each line is a short exact sequence. Then there is a unique Lie algebra homomorphism $\delta : \mathcal{L}(G) \to \mathcal{L}(G')$ such that $\delta \circ g = g' \circ \Psi$. 

 \end{lemma}
\begin{proof}
Since each line is a short exact sequence ,
\begin{center}
$\mathcal{L}(G) \cong {\mathcal{E}}/{\mathcal{L}(H)}$ and $\mathcal{L}(G') \cong {\mathcal{E}}/{\mathcal{L}(H')}$
\end{center}
Now define $\delta : \mathcal{L}(G) \to \mathcal{L}(G')$ by
\begin{center}
$\delta(x) = \delta (\bar{x} + \mathcal{L}(H)) = (g' \circ \Psi)(\bar{x})$ \quad for $\bar{x} \in \mathcal{E}$.
\end{center}
Since both $g'$ and $\Psi$ are Lie algebra homomorphisms, $\delta$ is also a Lie algebra homomorphism.\\
Also, for all $\bar{x} \in \mathcal{E}$,
\begin{equation}
	\begin{split}
		(\delta \circ g)(\bar{x}) &= \delta(g(\bar{x}))\\
								&= \delta(\bar{x} + \mathcal{L}(H))\\
								&= (g' \circ \Psi)(\bar{x})
	\end{split}
\end{equation}
\end{proof}
\begin{remark}
 Note that 
 \begin{center}
  $ 0 \to \mathbb{F} \xrightarrow{\mu} \mathfrak{gl}(V) \xrightarrow{\pi} \mathfrak{pgl}(V) \to 0$
 \end{center}
 is a central extension of $\mathfrak{pgl}(V)$ by $\mathbb{F}$ where $V$ is a $\mathbb{F}$-vector space, $\mu(k): v \mapsto kv$, the dialation map and $\pi$ is the canonical projection (see \cite{weibel}).
\end{remark}
\begin{theorem}\label{projrepplesken}
Consider the central extensions $0 \to \mathcal{L}(H) \xrightarrow{f} \mathcal{E} \xrightarrow{g} \mathcal{L}(G) \to 0$  and $0 \to  \mathbb{F} \xrightarrow{\mu} \mathfrak{gl}(V) \xrightarrow{\pi} \mathfrak{pgl}(V) \to 0$ of Lie algebras. Then for all pairs of Lie algebra homomorphisms $\Gamma : \mathcal{E} \to \mathfrak{gl}(V)$ and $\alpha : \mathcal{L}(H) \to \mathbb{F}$ such that $\mu \circ \alpha = \Gamma \circ f$, there is a projective representation $\Phi : \mathcal{L}(G) \to \mathfrak{pgl}(V)$ such that $\pi \circ \Gamma = \Phi \circ g$.

\[
\begin{tikzcd}
0 \arrow{r}&\mathcal{L}(H) \arrow{r}{f} \arrow{d}{\alpha} &
  \mathcal{E}   \arrow{r}{g} \arrow{d}{\Gamma} &
  \mathcal{L}(G) \arrow{r}    \arrow{d}{\Phi} &
  0
\\
0 \arrow{r}&\mathbb{F} \arrow{r}{\mu} &
 \mathfrak{gl}(V)      \arrow{r}{\pi} &
  \mathfrak{pgl}(V) \arrow{r} &
0
\end{tikzcd}
\]
\end{theorem}
\begin{proof}
 The proof follows directly from the Lemma \ref{projlemma}.
\end{proof}

\par The projective representation $\Phi$ of $\mathcal{L}(G)$ on the space $V$ is called \textit{irreducible} if $0$ and $V$ are the only $\Phi$-invariant subspaces of $V$.
\par Now let $Irr(\mathcal{E})$ denotes the set of all irreducible linear representations of $\mathcal{E}$ and  $Irr^{\alpha}(\mathcal{L}(G))$ denotes the set of all irreducible $\alpha$-representations of $\mathcal{L}(G)$ where $\alpha \in H^2(\mathcal{L}(G), \mathbb{C})$. The subsequent theorem establishes the correspondence between these sets.

\begin{theorem}
Let $\mathcal{E}$  a finite dimensional Lie algebra  and $\mathcal{L}(H)$ be a central subalgebra of $\mathcal{E}$ such that $\mathcal{E}$ is a cover of $\mathcal{L}(G)$. Then there is a bijection between the sets $Irr(\mathcal{E})$ and $\displaystyle\bigcup_{[\alpha]\in H^2(\mathcal{L}(G), \mathbb{C})} Irr^{\alpha}(\mathcal{L}(G))$.
\end{theorem}
\begin{proof}
\par Suppose that $\Phi$ is an irreducible $\alpha$-representation of $\mathcal{L}(G)$ on the vector space $V$. Then from Theorem \ref{linearrepcover}, we obtained a linear representation $\Gamma$ of $\mathcal{E}$ on $V$ and it is given by
\begin{center}
	$\Gamma (\hat{x}) = \Gamma(\hat{h}+s(\hat{y})) = \chi(\hat{h})I_V + \Phi(\hat{y})$
\end{center}
for $\hat{y} \in \mathcal{L}(G)$ and some linear map $\chi : \mathcal{L}(H) \to \mathbb{C}$.
\par Now it is enough to prove that $\Gamma$ is irreducible. For, assume that there is a non trivial proper subspace $W$ of $V$ which is $\Gamma$- invariant. That is, for all $\hat{x} \in \mathcal{E}$ and $w \in W$
\begin{equation}	
	\begin{split}
		\Gamma(\hat{x})(w) \in W & \Rightarrow (\chi(\hat{h})I_V + \Phi(\hat{y})(w) \in W \\
							   & \Rightarrow \chi(\hat{h})w + \Phi(\hat{y})(w) \in W \\
							   & \Rightarrow \chi(\hat{h})w + \Phi(\hat{y})(w) - \chi(\hat{h})w \in W \\
							   & \Rightarrow \Phi(\hat{y})(w) \in W \quad \text{ for all } \hat{y} \in \mathcal{L}(G)
	\end{split}
\end{equation}
That is, $\Phi$ is reducible, which is a contradiction. Therefore, $\Gamma$ is an irreducible linear representation of $\mathcal{E}$.
\par Conversely suppose that $\Gamma$ is an irreducible representation of $\mathcal{E}$. Then by Theorem \ref{projrepplesken}, there is a projective representation $\Phi : \mathcal{L}(G) \to \mathfrak{pgl}(V)$ and it is given by
\begin{center}
 $\Phi(\bar{x}) = \Gamma (x) + \mathcal{I}$ 
\end{center}
for some $x \in \mathcal{E}$, where $ \mathcal{I} = \{ k I_V : k \in \mathbb{F} \}$. 
\par To prove $\Phi$ is irreducible, suppose that $W$ be $\Phi$-invariant proper subspace of $V$. Then $\Phi(x)(w) \in W$ for every $x \in \mathcal{L}(G)$ and $w \in W$. \\
Now, for all $ w \in W$
\begin{equation}
	\begin{split}
			\Phi(\bar{x})(w) \in W  	&\Rightarrow (\Gamma (x) + \mathcal{I})(w) \in W \\
								&\Rightarrow \Gamma (x)(w) + kI_V(w) \in W\\
								&\Rightarrow  \Gamma (x)(w) + kI_V(w) - kI_V(w)\in W\\
								&\Rightarrow \Gamma(x)(w) \in W
	\end{split}
\end{equation}
That is, $\Gamma$ is reducible and it is a contradiction. Thus $\Phi$ is an irreducible projective representation of $\mathcal{L}(G)$.
\par Hence $\Phi \mapsto \Gamma$ gives the required bijection.
\end{proof}

\section*{Acknowledgment}
The first author thanks the Department of Science and Technology (DST) for the INSPIRE fellowship, which supported this work.

\end{document}